%% file: main.tex
\documentclass[12pt]{amsart}
\usepackage[letterpaper,margin=1in]{geometry}
\usepackage{amsmath,amssymb,amsthm}
\usepackage{graphicx}
\usepackage{enumitem}
\usepackage{microtype}
\usepackage{xcolor}
\usepackage[round]{natbib}
\usepackage[colorlinks,linkcolor=black,citecolor={blue!50!black},urlcolor={blue!80!black}]{hyperref}
\usepackage[capitalise,noabbrev]{cleveref}
\usepackage{subcaption}
\usepackage{braket}
\usepackage{booktabs}

\input{macros}

\makeatletter
\g@addto@macro\@setauthors{%
  \vspace{1.5em}%
  {\centering\normalsize\itshape Department of CS, ETH Z\"urich, Switzerland\par}%
  \vspace{0.5em}%
  {\centering\small\texttt{\{kfotopoulos, mhelcig, kdeck, jalsbach\}@ethz.ch}\par}}
\makeatother

\title[Non-Acceleration of the Heavy-Ball Method in 1D]%
      {Non-Acceleration of the Heavy-Ball Method in 1D}

\author{Konstantinos Fotopoulos}
\author{Michael Helcig}
\author{Karl Deck}
\author{Jannis Alsbach}

\date{}

\begin{document}

\begin{abstract}
\input{sections/abstract}
\end{abstract}

\maketitle

\markboth{\shorttitle}{\shorttitle}

\input{sections/intro}
\input{sections/setup}
\input{sections/analysis}
\input{sections/projected}
\input{sections/general}
\input{sections/connections}
\input{sections/conclusion}

\section*{Acknowledgments}
This work was conducted as part of the student seminar \textit{AI for Mathematics and Optimization} at ETH Z\"urich. We thank our supervising teaching assistant, Chung-En Tsai, for his valuable feedback and guidance throughout the seminar. We also acknowledge his independent, concurrent discovery of a four-dimensional counterexample similar to the one presented in Section 5. We thank Andrey Kharitenko for pointing out the connection between our one-dimensional cycling construction and the literature on absolute stability of Lur'e systems. We further acknowledge that Prof. Aymeric Dieuleveut and his collaborators independently and concurrently resolved the question of one-dimensional non-acceleration, as communicated to us privately.

KF acknowledges financial support for his master's studies through scholarships awarded by the Union of Greek Shipowners and the Onassis Foundation.

\section*{Declaration on usage of AI}

This work was conducted as part of the seminar \textit{AI for Mathematics and Optimization}, which explored the use of artificial intelligence in mathematical research. Throughout the project, we employed AI tools for mathematical exploration, proof development, and verification, combining human and AI-generated contributions.

The Fourier-based analysis, the derivation of the augmented cycling region, and the connection to the fractional sets were conceived and developed by the authors. The angular characterization in Lemma 3.1(ii) was conceived by the authors, initially proved with AI assistance, and subsequently established through a simpler, rigorous proof developed by the authors. Example 3.2 was generated by AI and independently verified by the authors. The rotation-and-projection construction of Theorem 4.1, including the observation that the ordering of the projected points remains unchanged under suitable rotations, originated with the authors. Its proof was initially developed with AI assistance and subsequently refined by the authors. The counterexample in Section 5 was generated by AI and verified by the authors, while the remaining results in that section were developed through iterative human–AI collaboration. The connection to Lur'e systems in Section 6 was identified by Andrey Kharitenko and elaborated by the authors.
The Lean 4 formalization was generated entirely by AI. The authors verified its correctness and its correspondence to the mathematical statements and proofs presented in the paper.

The authors take full responsibility for the mathematical correctness and presentation of the results.

\bibliographystyle{plainnat}
\bibliography{ref}

\appendix
\input{appendix/dft-interpolation}
\input{appendix/root-exclusion}

\end{document}

%% file: macros.tex
\theoremstyle{plain}
\newtheorem{theorem}{Theorem}[section]
\newtheorem{lemma}[theorem]{Lemma}
\newtheorem{corollary}[theorem]{Corollary}
\newtheorem{proposition}[theorem]{Proposition}
\newtheorem{conjecture}[theorem]{Conjecture}
\newtheorem{openproblem}[theorem]{Open Problem}

\theoremstyle{definition}
\newtheorem{definition}[theorem]{Definition}
\newtheorem{example}[theorem]{Example}

\theoremstyle{remark}

\crefname{openproblem}{Open Problem}{Open Problems}
\Crefname{openproblem}{Open Problem}{Open Problems}
\crefname{conjecture}{Conjecture}{Conjectures}
\Crefname{conjecture}{Conjecture}{Conjectures}

\newcommand{\R}{\mathbb{R}}
\newcommand{\Cx}{\mathbb{C}}
\newcommand{\Z}{\mathbb{Z}}

\newcommand{\interior}{\operatorname{int}}

\newcommand{\llb}{[\![}
\newcommand{\rrb}{]\!]}

\newcommand{\Fml}{\mathcal{F}_{\mu,L}}
\newcommand{\Qml}{\mathcal{Q}_{\mu,L}}

\newcommand{\Ocv}{\Omega_{\mathrm{cv}}}
\newcommand{\OKm}[2]{\Omega^{\circ}_{#1,#2}}
\newcommand{\Orou}{\Omega_{\circ\text{-}\mathrm{cycle}}}
\newcommand{\OKrou}[1]{\Omega_{#1\text{-}\circ\text{-}\mathrm{cycle}}}
\newcommand{\Oaug}{\Omega_{\mathrm{aug}\text{-}\circ}}

\newcommand{\Otvl}{\Omega_{\mathrm{TVL}}}
\newcommand{\Ocyc}[1]{\Omega^{(#1)}_{\mathrm{Cycle}}}
\newcommand{\OCycle}{\Omega_{\mathrm{Cycle}}}
\newcommand{\OKcyc}[2]{\Omega^{(#1)}_{#2\text{-}\mathrm{cycle}}}


%% file: sections/abstract.tex
Despite having an asymptotically optimal rate of convergence on convex quadratics, the acceleration of Polyak's heavy-ball method and a characterization of its dynamics within more general settings have remained long standing open problems. Recently, the work of Goujaud et al. [Math. Program. (2025)]
made progress regarding acceleration in the context of smooth, strongly convex minimization: for a small enough inverse condition number, either the method does not accelerate on quadratics, or there exists a smooth strongly convex function such that heavy-ball cycles over the roots of unity. The constructed functions have two variables, thus leaving the question of acceleration over univariate functions open. In this work, we use a Fourier-based analysis to i) provide a proof for the non-acceleration of the heavy-ball method over the class of univariate smooth strongly convex functions. We also ii) provide a strict enlargement of the previous root-of-unity cycling region, with insights regarding the ``fractional" sets of Goujaud et al. [arXiv:2502.19916], 
and iii) give a higher dimensional cycling instance which is not covered by root-of-unity cycling. 

%% file: sections/intro.tex
\section{Introduction}\label{sec:intro}

The efficiency of first-order methods in smooth, strongly convex optimization is governed by the condition number of the function to be minimized. Accelerated, momentum-type methods seek to mitigate the impact of the condition number on the rate of convergence compared to that of vanilla gradient descent. Among them, Polyak's heavy-ball method \citep{polyak1964some} is a canonical example, being asymptotically optimal for convex quadratic objectives. For a differentiable function \(f:\R^d \to \R\), initial points \((x_0, x_1)\), a step size \(\gamma\), and a momentum parameter \(\beta\), the heavy-ball update is given by
\begin{equation}
\label{eq:HB} \tag{HB}
x_{t+1} = x_t - \gamma \nabla f(x_t) + \beta(x_t - x_{t-1}). 
\end{equation}
On the class \(\Qml\) of \(L\)-smooth, \(\mu\)-strongly convex quadratics and under suitable choices of the parameters \((\gamma, \beta)\), \eqref{eq:HB} achieves the asymptotically optimal accelerated rate \(O\Big( \left(\frac{1-\sqrt{\kappa}}{1+\sqrt{\kappa}}\right)^t\Big)\), with \(\kappa := \mu / L\) the inverse condition number. 

The last rate matches the lower bound of \citet{nemirovski1983problem} for the class \(\Fml\) of general \(L\)-smooth, \(\mu\)-strongly convex functions. For this class of functions, however, the behavior of heavy-ball is not as clear. \citet{ghadimi2015global} prove convergence of \eqref{eq:HB} for a set of parameters \((\gamma, \beta)\) that does not achieve acceleration. \citet{lessard2016analysis} provide an explicit univariate function for which \eqref{eq:HB}, tuned optimally for quadratics, cycles over three points. Recently, \citet{goujaud2025provable} show that acceleration is not possible in dimension \(d \geq 2\): for every choice of parameters \((\gamma, \beta)\), either \eqref{eq:HB} does not accelerate on \(\Qml\), or there exists a \(d\)-variate function \(f \in \Fml(\R^d)\) for which \eqref{eq:HB} cycles over the \(K\)-th roots of unity for some \(K \geq 3\), traversed in order. Since these root-of-unity cycles are of affine dimension two, the proof does not directly apply to the subclass \(\Fml(\R) \subseteq \Fml\) of the univariate functions. The difficulty in dimension 1 stems from the fact that one-dimensional cycling does not allow for the symmetry existent in root-of-unity cycling when one applies smooth strongly convex interpolation conditions \citep{taylor2017smooth}. Whether the method accelerates on the class \(\Fml(\R)\) is stated as an open problem by \citet{goujaud2025provable, goujaud2025open}. \citet{goujaud2025provable} also provide an equivalent semi-definite program characterizing general cycling. With the dimension free, Section 6 therein also provides a cycling condition via an averaging argument. However, they do not give a \emph{closed-form} characterization of general cycling. Numerically, the parameter regions leading to one-dimensional and general cycling are indistinguishable from the region of root-of-unity cycling, suggesting equality of the three regions. 

Another observation includes the so-called ``fractional sets" of \citet{goujaud2025open}, which are obtained by plugging in fractional values of the cycle length into the quadratic equation characterizing root-of-unity cycling obtained in \citep{goujaud2025provable}. Numerically, these form the complement of the convergence region of \citet{taylor2018lyapunov}. 

Hence, non-acceleration in dimension 1, a closed-form characterization of general cycling, and the connection of cycling to the convergence set of \citet{taylor2018lyapunov} all remain open. To that end, we contribute as follows: 
\begin{enumerate}[label=(\roman*), leftmargin=2.4em, itemsep=3pt, topsep=3pt, parsep=0pt]
\item \emph{Augmented characterization and binding constraints.}
We consider the DFT decomposition of a general cycle to obtain an initial equivalent condition for cycling in two dimensions. For cycles over the roots of unity with \(m\)-jumps, the condition reduces equivalently to the non-negativity of a system of residuals. For \(m=1\), this recovers the region \(\Orou\) obtained by \citep{goujaud2025provable}, while the union over all \(m\) gives the strictly larger region \(\Oaug\). We argue that the ``fractional sets" of \cite{goujaud2025open} correspond to non-binding constraints of the above system. 
\item \emph{Projected cycles and 1D non-acceleration.} 
For every parameter pair \((\gamma, \beta)\) that admits a root-of-unity cycle, we construct a one-dimensional cycle by rotating the root-of-unity cycle by a carefully chosen phase and projecting it onto the first coordinate. Then, we prove existence of a function \(f \in \Fml(\R)\) on which \eqref{eq:HB} follows exactly this one-dimensional cycle. This proves that the existence of a root-of-unity cycle implies that of a one-dimensional cycle. Combined with the results of \citet{goujaud2025provable}, this shows that the method does not accelerate on \(\Fml(\R)\). 
\item \emph{General cycling.} By letting the dimension be free, we provide an example for which we have cycling that is covered by neither the \(\Orou\) region nor the \(\Oaug\) region. Instead, with the dimension being free, cycling is characterized by a mixed condition. We leave open the conjecture that \(\Ocyc{2} = \Oaug\), i.e., that with the dimension fixed, root-of-unity cycling is indeed characteristic of two-dimensional cycling. 
\end{enumerate}

Another closely related problem is the Kalman conjecture in absolute stability \citep{kalman1957physical}. There, one is interested in analyzing stability of a closed loop Lur'e system with slope restricted nonlinearities. The original conjecture is false, and \citet{seiler2021construction} give a sufficient condition for cycling of a SISO Lur'e system. Despite the similarity of the two problems, we are not aware of any existing works that apply the result of \citet{seiler2021construction} to one-dimensional heavy-ball cycling. To that end, 
we demonstrate how our construction relates with the one of \citet{seiler2021construction}. 

\paragraph{\textbf{Formal verification.}}
We accompany this paper with a formalization in Lean 4 and Mathlib of its main mathematical arguments (\href{https://github.com/kostfoto/Non-Acceleration-Heavy-Ball-1D.git}{https://github.com/kostfoto/Non-Acceleration-Heavy-Ball-1D.git}). The formalization includes the smooth strongly convex interpolation theorems, the Fourier reduction to the necessary mixed condition, and the characterization of primitive root-of-unity feasibility in Lemma~3.1, including the binding constraint and its angular characterization. We also verify Theorem~2.3, Example~3.2, and the one-dimensional cycle construction of Theorem~4.1, following the explicit sorting and chord-angle argument of the paper. Finally, we formally certify the counterexample of Section~5: all nine mixed inequalities are verified in exact arithmetic, and every primitive root-of-unity mode is excluded using the uniform large-$K$ argument of Appendix~B together with exact algebraic verification of the thirteen remaining finite cases. The deduction of quantitative non-acceleration relies on the cited result of \citet{goujaud2025provable}, while the sufficiency result of Proposition~5.1 is not formalized. 

%% file: sections/setup.tex
\section{Preliminaries and Problem Formulation}\label{sec:setup}

For $L>\mu>0$, let $\Fml(\R^d)$ be the set of differentiable functions $f:\R^d\to\R$ that are \(\mu\)-strongly convex and \(L\)-smooth:
\begin{align}
\tag{\(\mu\)-strong convexity}
f(y) &\geq f(x) + \langle \nabla f(x), y-x\rangle + \frac{\mu}{2} \|y - x\|^2 \\
\tag{\(L\)-smoothness}
f(y) &\leq f(x) + \langle \nabla f(x), y-x\rangle + \frac{L}{2} \|y - x\|^2.
\end{align}
Let \(\Fml = \bigcup_{d\geq 1} \Fml(\R^d)\) and let $\Qml\subset\Fml$ be the subset of such functions that are also quadratic. Given $(x_0,x_1)\in\R^d\times\R^d$ and parameters $(\gamma,\beta)\in\R^2$, the heavy-ball method \(\eqref{eq:HB}_{\gamma, \beta}\) on $f$ iterates the recursion~\eqref{eq:HB} from $(x_0,x_1)$. For a function class $\mathcal{F}$, let $\Ocv(\mathcal{F})$ be the parameters with which \eqref{eq:HB} converges for all $f\in\mathcal{F}$. For quadratics, we have \citep{polyak1964some,goujaud2025provable}
\begin{equation}\label{eq:cv-region}
\Ocv(\Qml)=\Big\{(\gamma,\beta):\ \lvert\beta\rvert<1,\ 0<\gamma<\tfrac{2(1+\beta)}{L}\Big\}.
\end{equation}
Throughout, when the function class \(\mathcal{F}\) is omitted, it is implied that we work on \(\Fml(\R^d)\), with \(\mu\), \(L\), and \(d\) clear from context. For \(N\in\Z_{\geq1}\) write \(\llb N\rrb:=\{0,\ldots,N-1\}\).

\begin{definition}[Cycle]\label[definition]{def:cycle}
Let $(\gamma,\beta)\in\Ocv(\Qml)$ and let $\{z_k\}_{k\in\llb K\rrb}\subseteq\R^d$ be points not all equal. For a differentiable function $f$, \(\eqref{eq:HB}_{\gamma, \beta}\) run on \(f\) \emph{cycles over} $\{z_k\}$ if the iterates of \eqref{eq:HB} starting from $(x_0,x_1)$ satisfy $x_t=z_{t\bmod K}$ for all \(t \in \Z_{\geq 0}\), and \(K\) is the minimum period for which this is true. We say \(\eqref{eq:HB}_{\gamma, \beta}\) cycles over $\{z_k\}$ \emph{on the class} $\mathcal{F}$ if such an $f\in\mathcal{F}$ exists, and denote by \(\OKcyc{d}{K}\) the set of all \((\gamma, \beta)\) leading to cycles of length \(K\) in dimension \(d\), by \(\Ocyc{d}=\bigcup_{K\geq 3} \OKcyc{d}{K}\) the set of all \((\gamma, \beta)\) leading to cycles of any length in dimension \(d\), and by \(\OCycle=\bigcup_{d\geq 1} \Ocyc{d}\) the set of all \((\gamma, \beta)\) leading to cycles of general length and dimension. 
\end{definition}

We let \(\theta_K := 2\pi/K\) and the inverse condition number \(\kappa := \mu/L \in (0, 1]\). Let $\OKrou{K}$ be the parameters with which \eqref{eq:HB} cycles over the \emph{ordered} \(K\)-th roots of unity $\{x_K^{k} = (\cos(\theta_K k), \sin(\theta_K k))^\top\}_{k\in\llb K\rrb}$ on $\Fml(\R^2)$; that is, the \(K\)-th roots of unity traversed in counterclockwise order. Let \(\Orou = \bigcup_K \OKrou{K}\). 

Throughout, the smooth strongly convex interpolation conditions of \citet{taylor2017smooth} will be useful:

\begin{theorem}[{Smooth strongly convex interpolation (reformulated), \citealp[Thm. 4]{taylor2017smooth}}]\label{thm:smooth-strongly-convex}
    For an index set \(I \subseteq \mathbb{N}\) and a set of pairs \(\{(x_i, g_i)\}_{i \in I} \subseteq \R^d \times \R^d\), there exists a function \(f \in \Fml(\R^d)\) for which \(\nabla f(x_i) = g_i\) for all \(i \in I\) if and only if there exist function values \(\{f_i\}_{i \in I} \subseteq \R\) such that for all \(i, j \in I\), 
    \[
    f_i - f_j - \langle g_j, x_i - x_j \rangle \geq \frac{1}{2(1-\mu/L)}\left(\frac{1}{L}\|g_i-g_j\|_2^2 + \mu \|x_i - x_j\|_2^2 - 2 \frac{\mu}{L} \langle g_i-g_j, x_i-x_j \rangle \right). 
    \]
    For the one-dimensional case, it is enough to consider the points in consecutive sorted order and require that \((g_i-g_j)/(y_i-y_j) \in [\mu, L]\). 
\end{theorem}

\noindent Note that for \(\mu = 0\) and \(L = \infty\), the conditions reduce to the convex interpolation conditions, which simply require that \(f_i \geq f_j + \langle g_j, x_i - x_j \rangle, \forall i, j \in I\). The proof of the general conditions involves normalization and conjugation steps, followed by an application of the convex interpolation conditions. We will often reapply these steps from scratch to derive simplified expressions. 

Using these conditions, \citet{goujaud2025provable} provided the following characterization of the \(\OKrou{K}\) cycling regions for \(K\geq3\) (the case of \(K=2\) can trivially be shown to require \(\gamma \geq 2(1+\beta)/L\), which is outside of \(\Ocv(\Qml)\)):

\begin{theorem}[{Ordered region, \citealp[Thm.~3.5]{goujaud2025provable}}]\label{thm:ordered}
For any $K\geq3$, the parameter region leading to the ordered \(K\)-th root-of-unity cycling is
\begin{equation*}
\OKrou{K}=\big\{(\gamma,\beta)\in\Ocv(\Qml):h(\mu\gamma,\cos\theta_K)\leq0\big\},
\end{equation*}
where $h(v,c)=v^2-2\big(\beta-c+\kappa(1-\beta c)\big)v+2\kappa(1-c)(1+\beta^2-2\beta c)$; equivalently, \(\langle M^\top(I-M) x_K^0, x_K^1 - x_K^0 \rangle \leq 0\) with 
\[
M = \frac{(1+\beta-\mu \gamma)I - R - \beta R^{-1}}{(L-\mu) \gamma}, \quad R = \begin{bmatrix}
    \cos \theta_k & -\sin \theta_k \\
    \sin \theta_k & \cos \theta_k
\end{bmatrix}.
\]
\end{theorem}

Before this work, we know the following inclusions: 
\[
\Orou \subseteq \Ocyc{d} \subseteq (\Ocv)^c \subseteq (\Otvl)^c, \quad d \geq 2,
\]
where \(\Otvl\) denotes the convergence region of \citet{taylor2018lyapunov}. The inverse inclusions are open. Together with the open inclusion \(\Orou \subseteq \Ocyc{1}\), they are the ones we study. In particular, from the work of \citet{goujaud2025open, goujaud2025provable}, the following open questions follow:
\begin{openproblem}\label[openproblem1]{op:oned}
Does $\Ocyc{1}=\Orou$? Does \eqref{eq:HB} accelerate on $\Fml(\R)$? Does $\OCycle=\Orou$? What is the conclusive relation between \(\Omega_{\mathrm{Cycle}}\), \(\Ocv\), and \(\Otvl\)?
\end{openproblem}

%% file: sections/analysis.tex
\section{Augmented regions and binding constraints}\label{sec:analysis}

We identify \(\R^2\cong\Cx\) via \((z_1, z_2) \mapsto z_1 + i z_2\), which simplifies the analysis of the rotations. For \(K\geq3\), write \(\omega_K:=e^{2\pi i/K}\) and let
\[
U_K := \{m \in \{1, \ldots, K-1\} : \gcd(m, K) = 1\}. 
\]
For \(m \in U_K\), we define the unnormalized and normalized \eqref{eq:HB} operators \(\hat M_m\) and \(M_m\) respectively as 
\begin{equation}\label{eq:mode-operator}
\hat M_m = \frac{(1+\beta) - \omega_K^m - \beta \omega_K^{-m}}{\gamma} \in \Cx. \quad M_m = \frac{(1+\beta - \mu\gamma) - \omega_K^m - \beta \omega_K^{-m}}{(L-\mu)\gamma} \in \Cx. 
\end{equation}
For \(m=1\), \(M_m\) corresponds to the operator used in \cref{thm:ordered}, written using complex numbers. Writing \(M_m=p_m-iq_m\), we have
\begin{equation}\label{eq:pq}
\begin{aligned}
p_m&=\frac{(1+\beta-\mu\gamma)-(1+\beta)\cos(m\theta_K)}{(L-\mu)\gamma},\\
q_m&=\frac{(1-\beta)\sin(m\theta_K)}{(L-\mu)\gamma},
\end{aligned}
\end{equation}
and we drop \(m\) when clear. 

Our first main result is that the ordered root-of-unity cycles do not capture all the possible cycling behavior over roots of unity. For \(m\in U_K\), let \(\OKm{K}{m}\) be the set of parameters in \(\Ocv(\Qml)\) with which \eqref{eq:HB} cycles over the \(m\)-jump \(K\)-th roots of unity cycle
\[
    1,\omega_K^m,\omega_K^{2m},\ldots,\omega_K^{(K-1)m}.
\]
Thus \(\OKm{K}{1}=\OKrou{K}\) is the ordered root-of-unity cycling region of \cite{goujaud2025provable}. Note that for \(K, m\) with \(\gcd(m, K) \neq 1\) the points of the cycle repeat, and we get cycling over a smaller cycle of length \(K'=K/\gcd(m, K)\) (formally, the generator \(\omega_K^m\) of the cyclic subgroup has order smaller than \(K\).) This justifies why we focus on \(K, m\) with \(m \in U_K\). We define the augmented root-of-unity cycling region by
\begin{equation*}
    \Oaug:=\bigcup_{K\geq3}\ \bigcup_{m\in U_K}\OKm{K}{m}. 
\end{equation*}
In the following, we provide a characterization of the augmented root-of-unity cycling region, and then show that it is strictly larger than the ordered one.

\paragraph{\textbf{Fourier reduction.}}
Let \(\{z_t\}_{t=0}^{K-1}\subseteq \Cx\) be a two-dimensional cycle of length \(K\) and write its discrete Fourier decomposition as
\begin{equation}\label{eq:dft}
    z_t=\sum_{m=0}^{K-1}c_m\omega_K^{mt}, \quad c_m = \frac{1}{K} \sum_{t=0}^{K-1} z_t \omega_K^{-mt}. 
\end{equation}
\eqref{eq:HB} cycles over the points \(\{z_t\}\) for a function \(f \in \Fml(\R^2)\) if and only if for all \(t\) we have that
\begin{equation}\label{eq:grad_equiv}
    \nabla f(z_t) = \frac{(1+\beta)z_t - z_{t+1} - \beta z_{t-1}}{\gamma} = \sum_{m} c_m \frac{(1+\beta) - \omega_K^m - \beta \omega_K^{-m}}{\gamma} \omega_K^{mt} = \sum_{m} c_m \hat M_m \omega_K^{mt}
\end{equation}
Following \citet{goujaud2025provable}, we apply the smooth, strongly convex interpolation conditions to the points of the cycle as in \cref{thm:smooth-strongly-convex}. The argument is provided in Appendix~\ref{app:dft-interpolation}. 
The interpolation condition between the indices \(j \to j+\Delta\) reduces to
\begin{align}\label{eq:dft-inner-product-main}
f_{j+\Delta} &\geq f_j + \operatorname{Re}\left[
    \left(\sum_{m'} c_{m'}(1- M_{m'})\omega_K^{m'j}\right)^*
    \left(\sum_m c_m  M_m\big(\omega_K^{m(j+\Delta)}-\omega_K^{mj}\big)\right)
\right] \notag\\
& = f_j+
\sum_m |c_m|^2\phi_{K,m}(\Delta)
\\
& \qquad + \sum_{m\neq m'}\operatorname{Re}\left[
    c_{m'}^*c_m(1- M_{m'})^* M_m
    (\omega_K^{m\Delta}-1)\omega_K^{(m-m')j}
\right] \notag,
\end{align}
where for the non-cross term of mode \(m\) we define the residual 
\begin{equation}\label{eq:interpolation-residual}
    \phi_{K,m}(\Delta)
    :=\operatorname{Re}\!\left[(1- M_m)^* M_m(\omega_K^{m\Delta}-1)\right], \quad \Delta \in \{0, \ldots, K-1\}.
\end{equation}
For a fixed constraint \(\Delta\), the residual captures how much slack each mode \(m\) has in the constraint. A cycle exists if and only if the above holds for every \(j, \Delta \in [\![K]\!]\). For \(m\neq m'\), we have \(\sum_j \omega_K^{(m-m')j}=0\), and summing over \(j\) eliminates all \(f_j\) and all cross terms. Hence, every cycle of length \(K\) in 2D satisfies the necessary mixed condition
\begin{equation}\label{eq:mixed-main}
    \sum_{m=0}^{K-1}s_m\phi_{K,m}(\Delta)\leq 0,
    \qquad \Delta=1,\ldots,K-1,
\end{equation}
for some weights \(s_m\propto |c_m|^2\geq0\), \(\sum_m s_m=1\). The constraint for \(\Delta=0\) can be ignored since it reduces to \(0\leq 0\) and is always true. This can also be seen from \eqref{eq:dft-inner-product-main}, where \(\Delta=0\) corresponds to \(j \to j\) interpolation conditions that are trivially always true. With this, we can state our first lemma: 

\begin{lemma}[Primitive R-o-U feasibility]\label{thm:primitive-rou-feasibility}
Fix \(K\geq 3\) and \(m\in U_K\). Then
\begin{equation}\label{eq:primitive-region-residual}
    \Omega_{K,m}^{\circ}
    =\left\{(\gamma,\beta)\in\Omega_{\text{cv}}(\mathcal Q_{\mu,L}):
    \phi_{K,m}(\Delta)\leq 0 \text{ for every } \Delta=1,\ldots,K-1\right\}.
\end{equation}
Moreover, assume without loss of generality that \(1\leq m\leq (K-1)/2\) and write \( M_m=p_m-iq_m\). Then, \(q_m>0\) and the following are equivalent:
\begin{enumerate}[label=(\roman*), nosep]
    \item \((\gamma,\beta)\in \Omega_{K,m}^{\circ}\);
    \item \(p_m(1-p_m)>q_m^2\), \(M_m\) and \(1-M_m\) lie in the fourth and first quadrants respectively, and
    \begin{equation}\label{eq:angle-bound-main}
        -\operatorname{arg} M_m+\operatorname{arg}(1-M_m)\leq \frac{\pi}{K},
    \end{equation}
    where \(\arg:\Cx\to(-\pi,\pi]\) is the principal argument.
\end{enumerate}
When these conditions hold, the binding constraint corresponds to the unique \(\Delta_m \equiv m^{-1} \bmod K \). If \(m > (K-1)/2\) then \(\Delta_m = \Delta_{K-m}\). For the interior of \(\Omega_{K,m}^{\circ}\) all inequalities become strict. 
\end{lemma}

\begin{proof}[Proof of Lemma~\ref{thm:primitive-rou-feasibility}]

For the first part of the statement, notice that the \(m\)-th mode root-of-unity cycle corresponds to Fourier coefficients \(c_m\in \Cx\) and \(c_{m'} = 0\) for all \(m' \neq m\). In this case, \(c_{m'}^*c_m=0\), there are no cross terms, and the same reduction becomes sufficient by choosing all function values to be \(0\). In addition, the distribution \(\{s_m\}_{m=0}^{K-1}\) is a discrete Dirac distribution over \(m\). Hence, cycling reduces equivalently to the system \(\{\phi_{K, m}(\Delta) \leq 0\}_{\Delta=1}^{K-1}\).

We first assume that \(m \leq (K-1)/2\). Since \(|\beta|<1\), from \eqref{eq:pq} it follows immediately that \(q_m > 0\) because \(m \theta_K \in (0, \pi)\). Since \(\gcd(m, K)=1\), for \(\Delta \neq 0\) we have \(m\Delta \neq 0 \bmod K\) and \(1-\cos(m\Delta\theta_K) > 0\). Expanding the residual then gives
\begin{align}\label{eq:phi-trig-main}
    \phi_{K,m}(\Delta)&=\big(p_m(1-p_m)-q_m^2\big)(\cos(m\Delta\theta_K)-1)+q_m\sin(m\Delta\theta_K) \notag \\
    &= \underbrace{(1-\cos(m\Delta\theta_K))}_{>0} \left(-\big(p_m(1-p_m)-q_m^2\big) + q_m \frac{\sin(m\Delta\theta_K)}{1-\cos(m\Delta\theta_K)} \right). 
\end{align}
The system \(\{\phi_{K, m}(\Delta) \leq 0\}_{\Delta=1}^{K-1}\) is feasible iff the second factor is non-positive for every \(\Delta = 1, \ldots, K-1\), and 
\[
\frac{\sin(m\Delta\theta_K)}{1-\cos(m\Delta\theta_K)} = \frac{2\sin(m\Delta\theta_K/2) \cos(m\Delta\theta_K/2)}{2\sin^2(m\Delta\theta_K/2)} = \frac{\cos(m\Delta\theta_K/2)}{\sin(m\Delta \theta_K/2)} = \cot(m\Delta \theta_K/2)
\]
is maximized for \(\Delta_m \equiv m^{-1} \pmod K\), for which the angle \(m\Delta\theta_K \bmod 2\pi\) is minimized in \((0, \pi]\). In other words, the whole system reduces to a single inequality \(\phi_{K,m}(\Delta_m)\leq 0\). For this \(\Delta = \Delta_m\), the term \(q_m \frac{\sin(m\Delta\theta_K)}{1-\cos(m\Delta\theta_K)}\) is positive, and thus feasibility implies that \(p_m(1-p_m)-q_m^2 > 0\), which in turn implies that \(p_m \in (0, 1)\). Together with \(q_m>0\), this gives that \(\arg M_m \in (-\pi/2, 0)\) and \(\arg (1-M_m) \in (0, \pi/2)\). Returning back to the original definition of the residual, we have that \((\gamma, \beta) \in \OKm{K}{m}\) if and only if
\begin{align*}
& \phi_{K, m}(\Delta_m) = \operatorname{Re}\!\left[(1- M_m)^* M_m(\omega_K-1)\right] \leq 0 \\
\Leftrightarrow \ & \operatorname{Im}\left[(1- M_m)^* M_m(\omega_K-1)e^{\pi i / 2}\right] \leq 0 \\
\Leftrightarrow \ & \operatorname{Im}\left[(1- M_m)^* M_m e^{i(\pi + \pi/K)}\right] \leq 0 \\
\Leftrightarrow \ & \operatorname{Im}\left[((1- M_m) M_m^*) e^{i(\pi - \pi/K)}\right] \geq 0 
\end{align*}
Because \(\arg M_m^* = -\arg M_m \in (0, \pi/2), \ \arg(1-M_m) \in (0, \pi/2)\), we have \(\arg((1-M_m)M_m^*) \in (0, \pi)\), and the last inequality is possible if and only if 
\begin{gather*}
\arg((1-M_m) M_m^*) + \pi - \frac{\pi}{K} \leq \pi \\ 
\Leftrightarrow -\operatorname{arg} M_m+\operatorname{arg}(1-M_m)=\arg((1-M_m) M_m^*) \leq \frac{\pi}{K}, 
\end{gather*}
Finally, by keeping track of the inequalities in the proof, it follows that strict feasibility is equivalent to a strict angle bound. 

Now, if \(m > (K-1)/2\), then we can set \(m \leftarrow K-m\) and everything remains the same. Indeed, setting \(m \leftarrow K-m\) reflects the cycle about the real line. The function class \(\Fml(\R^2)\) is closed under reflections, and for a reflected function \(f' \in \Fml(\R^2)\), \eqref{eq:HB} run on \(f'\) yields the reflected iterates. Formally, with \(m > (K-1)/2\) we have \(q_m<0\), and thus the binding \(\Delta_m = \Delta_{K-m}\), for which \(m\Delta_{K-m} \equiv -1 \bmod K\), \(\sin(m\Delta_{K-m}\theta_K) = -\sin(\theta_K) < 0\) and \(q_m \sin(m\Delta_{K-m}\theta_k) / (1-\cos(m\Delta_{K-m}\theta_K)) > 0\) is maximized. 
\end{proof}

\paragraph{\textbf{Binding constraints and fractional sets.}}
For \(m=1\), the binding residual corresponds to \(\Delta=1\), and the lemma recovers the characterization of \cref{thm:ordered}. For \(m \neq 1\), checking \(\phi_{K,m}(1)\) is not enough. We show in the following example that \(\Oaug\) is strictly larger than \(\Orou\) in general, by providing an explicit instance where for a pair \((\gamma, \beta)\) we have \((\gamma, \beta) \notin \Orou\) but \((\gamma, \beta) \in \interior(\OKm{5}{2})\). The augmented regions can be seen visually in \cref{fig:augmented}.

\begin{figure}[t]
    \centering
    \begin{subfigure}{0.45\linewidth}
    \includegraphics[width=0.9\linewidth]{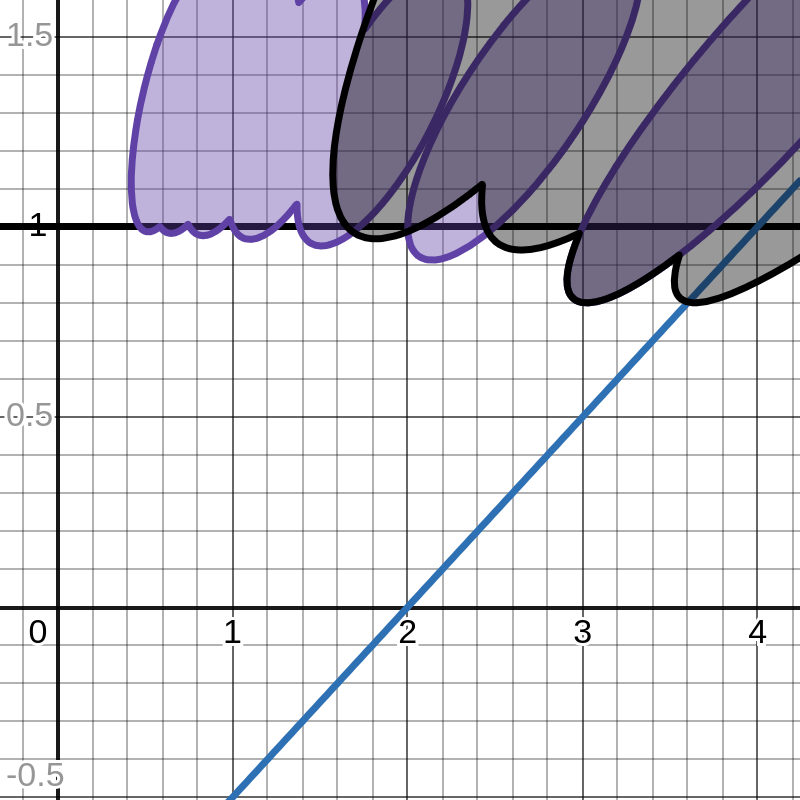}
    \caption{\(\kappa = 0.8\). Root-of-unity regions with \(m=1\) (purple) and \(m=2\) (black), plotted using the characterization in Lemma~\ref{thm:primitive-rou-feasibility}. The black region is not covered by the ordered \(m=1\) root-of-unity region and belongs to the augmented region.}
    \label{fig:augmented}
    \end{subfigure}
    \hfill
    \begin{subfigure}{0.45\linewidth}
    \centering
    \includegraphics[width=0.9\linewidth]{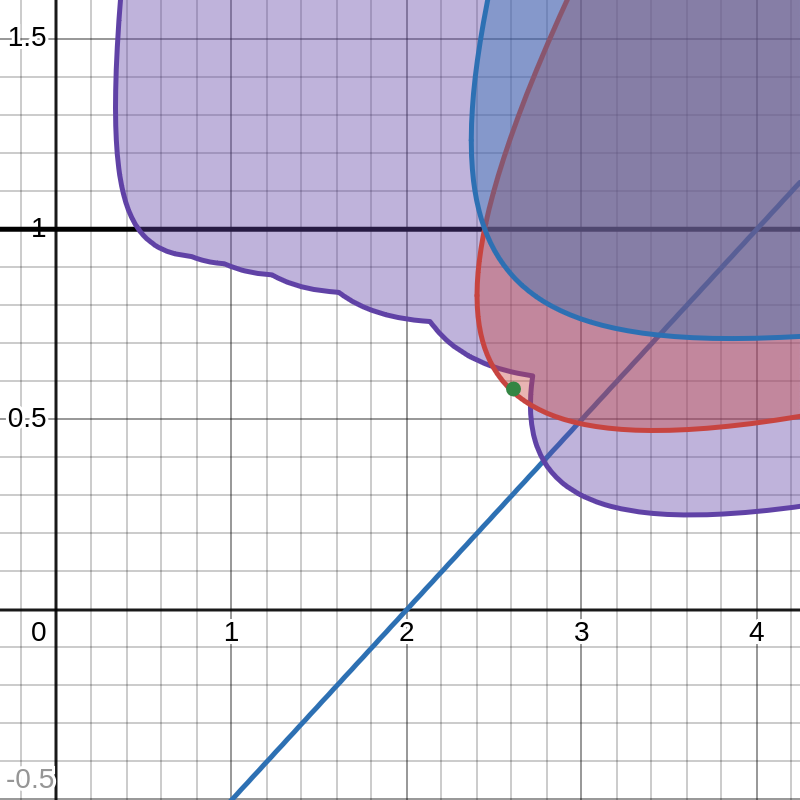}
    \caption{\(\kappa = 0.3\). The red region corresponds to the fractional condition obtained from the \(\Delta=1\) residual of \((K,m)=(7,2)\). The blue region is the stricter genuine cycling region obtained by checking the maximal residual. Hence fractional regions are not automatically cycling certificates.}
    \label{fig:rational}
    \end{subfigure}
    \caption{Augmented regions and binding constraints}
    \label{fig:aug-bind}
\end{figure}

\begin{example}\label{ex:witness}
Let $\kappa = 4/5$, $\beta = 17/20$, and $L\gamma = 18/5$, so that $\mu\gamma = 72/25$ and $(\gamma,\beta)\in\Ocv(\Qml)$ by \eqref{eq:cv-region}. For $(K,m)=(5,2)$, condition (ii) of Lemma~\ref{thm:primitive-rou-feasibility} holds strictly:
\begin{equation*}
    p_2(1-p_2)-q_2^2 = \frac{17215\sqrt{5}-36285}{10368} > 0
    \quad\text{and}\quad
    \tan\tfrac{\pi}{5}\,\big(p_2(1-p_2)-q_2^2\big) - q_2 > 0.
\end{equation*}
Since $p_2$ and $q_2$ are continuous in $(\gamma,\beta)$ on the open set $\Ocv(\Qml)$, the subset of $\Ocv(\Qml)$ defined by the two strict inequalities is open and contained in $\OKm{5}{2}$, so $(\gamma,\beta)\in\interior(\OKm{5}{2})\subseteq\interior(\Oaug)$. Moreover, the minimum of the quadratic \(h\) of \cref{thm:ordered} with $v = 72/25$ over $c\in[-1/2,1]$ equals $63/500 > 0$; as $\cos\theta_{K'}\in[-1/2,1)$ for every $K'\geq 3$, \cref{thm:ordered} gives $(\gamma,\beta)\notin\OKm{K'}{1}$ for every $K'\geq 3$, that is, $(\gamma,\beta)\notin\Orou$. Hence
\begin{equation*}
    \interior(\Oaug)\setminus\Orou\neq\emptyset.
\end{equation*}
\end{example}

Notice that the expression of \(\phi_{K, m}(\cdot)\) depends only on \(m / K\) as \(\phi_{K, m}(\cdot) = \phi_{m/K}(\cdot)\). The fractional regions discussed in \cite{goujaud2025open} are obtained by plugging a fractional \(K' = K / m\) into the \(m=1\) characteristic equation; that is, they correspond to \(\phi_{K', 1}(1) = \phi_{m/K}(1) = \phi_{K, m}(1) \leq 0\). As we have seen already, for \(m\neq1\) the latter does not correspond to the binding constraint, and it remains unclear what types of behavior the fractional set corresponds to; see \cref{fig:rational}. 

%% file: sections/projected.tex
\section{Projected cycles and 1D non-acceleration}\label{sec:projected}

In this section, we show that \eqref{eq:HB} does not accelerate on \(\Fml(\R)\). The main idea is to construct a 1D cycle by rotating a root-of-unity cycle by some well-chosen angle \(\phi_\ast\) and projecting it into onto \(\mathbb{R}\). The angle must be chosen in a way that takes the phase induced by the \eqref{eq:HB} operator \(\hat M_m\) into consideration, and allows control of the cross-terms of \cref{eq:dft-inner-product-main}. We will apply the characterizations in Lemma~\ref{thm:primitive-rou-feasibility}. Note that DFT can also describe 1D cycles: cycles over the real line correspond to Fourier coefficients \(c_{K-m}^* = c_m\).

For 1D, according to \cref{thm:smooth-strongly-convex}, the interpolation conditions reduce to gradient slope conditions between consecutive points. After projection, these do not necessarily correspond to the same binding pairs \((j, j+\Delta_m)\) of root-of-unity cycling, meaning that the non-cross terms of Eq. \eqref{eq:dft-inner-product-main} have additional ``slack" that can help correct the cross terms. We reformulate the interpolation conditions based on their simplified one-dimensional form, and then it suffices to find a rotation that does not change the sort order of the projected points. The existence of such a rotation is ensured by the angle bound (ii) of Lemma~\ref{thm:primitive-rou-feasibility}.

\begin{theorem}[1D cycle construction]\label{thm:projection-main}
Let \(K\geq 3\), \(m\in U_K\), and \((\gamma,\beta)\in\Omega_{K,m}^{\circ}\). Then \((\gamma,\beta)\in\Omega_{K\text{-cycle}}^{(1)}\). More explicitly, after possibly replacing \(m\) by \(K-m\), set
\begin{equation}\label{eq:phase-main}
    \varphi_*:=-\frac12\operatorname{arg}\big((1-M_m)M_m\big)-\frac{\pi}{2K}.
\end{equation}
Define one-dimensional points and gradients by
\begin{equation}\label{eq:oned-cycle-data-main}
    y_t:=2\cos(mt\theta_K+\varphi_*),\quad
    g_t:=\frac{(1+\beta)y_t-y_{t+1}-\beta y_{t-1}}{\gamma},
    \quad t=0,\ldots,K-1,
\end{equation}
with indices interpreted modulo \(K\). Then, after sorting the points \(y_t\) increasingly, every adjacent slope lies in \([\mu,L]\). Therefore there exists \(f\in\mathcal F_{\mu,L}(\mathbb R)\) with \(f'(y_t)=g_t\) for all \(t\), and HB initialized at \((y_0,y_1)\) cycles on these \(K\) points.
\end{theorem}

\begin{proof}[Proof of \cref{thm:projection-main}]
We let any rotation phase \(\varphi\) and define
\[
    y_t=2\cos(mt\theta_K+\varphi) = 2 \operatorname{Re}( e^{i\varphi}\omega_K^{mt} ),
    \qquad
    g_t=\frac{(1+\beta)y_t-y_{t+1}-\beta y_{t-1}}{\gamma}.
\]
For \(\Delta\) such that \(y_{j+\Delta}>y_j\), set \(w_{j,\Delta}:=e^{i\varphi}\big(\omega_K^{m(j+\Delta)}-\omega_K^{mj}\big)\). Then 
\[
y_{j+\Delta}-y_j=2\operatorname{Re}w_{j,\Delta}>0
\]
and the \eqref{eq:HB} update gives
\[
    \gamma(g_{j+\Delta}-g_j)=2\operatorname{Re}\left(((1+\beta)-\omega_K^m-\beta\omega_K^{-m})w_{j,\Delta}\right).
\]
In 1D, from Theorem~\ref{thm:smooth-strongly-convex}, interpolation reduces to slope conditions for consecutive sorted pairs. The slope \((g_{j+\Delta}-g_j)/(y_{j+\Delta}-y_j)\) lies in \([\mu,L]\) if and only if
\begin{align*}
    &\frac{\operatorname{Re}\left(((1+\beta)-\omega_K^m-\beta\omega_K^{-m})w_{j,\Delta}\right)}{\gamma} - \mu \operatorname{Re}w_{j, \Delta} \geq 0, \text{ and} \\
    &\frac{\operatorname{Re}\left(((1+\beta)-\omega_K^m-\beta\omega_K^{-m})w_{j,\Delta}\right)}{\gamma} - L \operatorname{Re}w_{j, \Delta} \leq 0.
\end{align*}
The first can be rewritten as
\begin{align*}
    & \frac{\operatorname{Re}\left(((1+\beta)-\omega_K^m-\beta\omega_K^{-m})w_{j,\Delta}\right)}{\gamma} - \mu \operatorname{Re}w_{j, \Delta} \geq 0 \\ \Leftrightarrow \ & (L-\mu) \operatorname{Re} \left(\frac{(1+\beta-\mu\gamma)-\omega_K^m-\beta\omega_K^{-m})}{(L-\mu)\gamma}w_{j,\Delta}\right) \\
    & \qquad = \operatorname{Re} \left(\frac{(1+\beta)-\omega_K^m-\beta\omega_K^{-m})w_{j,\Delta}}{\gamma} - \mu w_{j, \Delta}\right) \geq 0 \\
    \Leftrightarrow \ & \operatorname{Re}(M_m w_{j, \Delta}) \geq 0, 
\intertext{while the second can be written as}
    & \frac{\operatorname{Re}\left(((1+\beta)-\omega_K^m-\beta\omega_K^{-m})w_{j,\Delta}\right)}{\gamma} - L \operatorname{Re}w_{j, \Delta} \leq 0 \\ 
    \Leftrightarrow \ & \frac{\operatorname{Re}\left(((1+\beta-\mu\gamma)-\omega_K^m-\beta\omega_K^{-m})w_{j,\Delta}\right)}{\gamma} - (L-\mu) \operatorname{Re}w_{j, \Delta} \leq 0,\\
    \Leftrightarrow \ & (L-\mu) \operatorname{Re} \left(\left(\frac{(1+\beta-\mu\gamma)-\omega_K^m-\beta\omega_K^{-m})}{(L-\mu)\gamma} - 1\right)w_{j,\Delta}\right) \leq 0 \\
    \Leftrightarrow \ & \operatorname{Re}((1-M_m) w_{j, \Delta}) \geq 0. 
\end{align*}
Hence, we have interpolation if and only if
\begin{align*}
    &\operatorname{Re}(M_mw_{j,\Delta})\geq0,
    \qquad
    \operatorname{Re}((1-M_m)w_{j,\Delta})\geq0 \\
    \Leftrightarrow \ &
    \operatorname{arg}w_{j,\Delta}\in
    \left[-\frac{\pi}{2}-\arg M_m, \frac{\pi}{2} - \arg M_m \right]\cap\left[-\frac{\pi}{2}-\arg (1- M_m), \frac{\pi}{2}-\arg (1- M_m)\right] \\
    \Leftrightarrow \ &
    \operatorname{arg}w_{j,\Delta}\in
    \left[-\frac{\pi}{2}-\arg M_m,\,\frac{\pi}{2}-\arg (1- M_m)\right],
\end{align*}
where in both equivalences we used that \(\arg M_m \in (-\pi/2, 0)\) and \(\arg(1-M_m) \in (0, \pi/2)\) given from Lemma~\ref{thm:primitive-rou-feasibility}, so that the arguments always remain principal. 

We now determine the sort-order. We sort the numbers \(\cos(\ell\theta_K+\varphi)\), \(\ell=0,\ldots,K-1\). For \(\varphi\in[-\pi/K,0]\), their order does not change, because two such cosines coincide only when \(\varphi\) is a multiple of \(\pi/K\). Taking \(\varphi<0\) close to zero gives the explicit order
\[
\begin{array}{ll}
K=2r+1: & r+1,\ r,\ r+2,\ r-1,\ldots,\ 2r,\ 1,\ 0,\\[1mm]
K=2r: & r,\ r+1,\ r-1,\ r+2,\ldots,\ 2r-1,\ 1,\ 0.
\end{array}
\]
In both cases, the sum of the indices of two consecutive sorted points is congruent to either \(0\) or \(1\) modulo \(K\), and thus
\begin{equation}\label{eq:two-chord-args}
    \operatorname{arg}w_{j, \Delta}\in
    \left\{\frac{\pi}{2}+\varphi,\; -\frac{\pi}{2}+\frac{\pi}{K}+\varphi\right\}.
\end{equation}
Therefore all adjacent slopes are in \([\mu,L]\) if and only if
\begin{equation}\label{eq:valid-phase-interval}
    -\arg M_m-\frac{\pi}{K}\leq \varphi\leq -\arg(1- M_m).
\end{equation}
The interval is nonempty by Lemma~\ref{thm:primitive-rou-feasibility} (ii). We choose the mid-point \(\varphi_*=-\frac12\operatorname{arg}((1-M_m)M_m)-\frac{\pi}{2K}\). Hence, rotation does not ruin sort-order, and every adjacent sorted slope belongs in \([\mu,L]\) and the 1D interpolation conditions hold.
\end{proof}

From \cref{thm:projection-main} it follows immediately that root-of-unity cycling implies one-dimensional cycling for the same parameter pair \((\gamma, \beta)\). Hence, we can use the proof of non-acceleration of \citet[Theorem 3.6]{goujaud2025provable} verbatim in the one-dimensional case as well. We provide the following corollary:

\begin{corollary}[1D cycling and non-acceleration]\label{cor:nonaccel-main}
We have
\[
    \Oaug=\bigcup_{K, m} \OKm{K}{m}\subseteq \Ocyc{1}
\]
and heavy-ball does not accelerate on \(\Fml(\R)\). More precisely, with the same absolute constant \(C>50/3\) as in \citep{goujaud2025provable}, for every \(0<\mu<L\) and every choice of parameters \((\gamma,\beta)\), either heavy-ball already has quadratic worst-case rate larger than \((1-C\kappa)/(1+C\kappa)\), or there exists a function \(f\in\Fml(\R)\) and an initialization for which heavy-ball cycles. Hence no parameter choice gives an accelerated worst-case rate on the one-dimensional class.
\end{corollary}

%% file: sections/general.tex
\section{General Cycling}\label{sec:general}

Throughout this section, we set \(L=1\) (accumulating it into \(\gamma\) WLOG) and \(\mu=\kappa\in(0,1)\). In \cref{sec:analysis}, we saw that the dependence of all involved quantities is on the ratio \(2\pi m/K\), where \(K\) is the cycle length and \(m\) the mode. A natural way to remove the dependence from \(K\) is to consider the effective frequency \(f=m/K\). 

We repeat the Fourier-based analysis for frequencies \(f \in [0, 1)\) and \(d\)-dimensional cycles. We define the normalized heavy-ball operator
\begin{equation}
M(f)=\frac{(1+\beta-\kappa \gamma) - e^{2\pi i f} - \beta e^{-2\pi i f}}{(1-\kappa)\gamma} = p(f) - i q(f),  
\label{eq:norm-operator-cont}
\end{equation}
\begin{equation}
    p(f) = \frac{(1+\beta-\kappa \gamma) - (1+\beta)\cos (2\pi f) }{(1-\kappa)\gamma}, \quad q(f) = \frac{(1-\beta)\sin (2\pi f)}{(1-\kappa)\gamma}. 
    \label{eq:pq-cont}
\end{equation}
We view a \(d\)-dimensional \(K\)-cycle as a signal \(\{z_t\}_{t=0}^{K-1} \subseteq \Cx^{\lceil d/2 \rceil}\), where we pad with an additional zeroed dimension in case of odd \(d\). The higher-dimensional DFT decomposition of this signal allows us to write it as
\begin{equation}
    z_t = \sum_{m = 0}^{K-1} c_{m/K} \omega_K^{m t}, \quad \{c_{m/K}\}_{m=0}^{K-1} \subseteq \Cx^{\lceil d/2 \rceil}. 
    \label{eq:dft-cont}
\end{equation}
For \(K \geq 2\), we have that \(c_{m/K} \neq 0\) for at least one \(m \geq 1\), for otherwise we would have a constant cycle of length \(K=1\). Note that, as mentioned already, cycling with \(K=2\) is impossible for \(\gamma < \frac{2(1+\beta)}{L}\), and thus we restrict to \(K\geq 3\). 

With a similar analysis to that of the 2-dimensional case, heavy-ball cycles over these points if and only if for every \(j, j + \Delta\) we have
\begin{align}\label{eq:dft-inner-product-main-cont}
f_{j+\Delta} &\geq f_j + \operatorname{Re}\left[
    \left(\sum_{m'} c_{m'/K}(1- M(m'/K))\omega_K^{m'j}\right)^*
    \left(\sum_m c_{m/K}  M(m/K)\big(\omega_K^{m(j+\Delta)}-\omega_K^{mj}\big)\right)
\right] \notag\\
& = f_j+
\sum_m \|c_{m/K}\|^2\phi(m/K, \Delta)
\\
& \qquad + \sum_{m\neq m'}\operatorname{Re}\left[
    (c_{m'/K}^*c_{m/K})(1- M(m'/K))^* M(m/K)
    (\omega_K^{m\Delta}-1)\omega_K^{(m-m')j}
\right] \notag,
\end{align}
where for the non-cross term of mode \(m\) we defined the residual 
\begin{equation}\label{eq:interpolation-residual-cont}
    \phi(f, \Delta)
    :=\operatorname{Re}\!\left[(1 - M(f))^* M(f)(e^{2\pi i f \Delta}-1)\right], \quad \Delta \in \{0, \ldots, K-1\}.
\end{equation}
After summing over \(j\), the cross-terms go to \(0\) and we are left with a necessary condition for cycling:
\begin{equation}
    \sum_m s_{m/K} \phi(m/K, \Delta) \leq 0
    \label{eq:mixed-main-cont}, \quad \forall \Delta \in \{0, \ldots, K-1\}
\end{equation}
with \(s_{m/K} \propto \|c_{m/K}\|^2, \ \sum_{m=0}^{K-1} s_{m /K} = 1\), and \(s_0 < 1\). 

This is a mixture of the pure conditions based on the finite distribution \(s(f) = \sum_{m=0}^{K-1} s_{m/K} \delta_{m/K}(f)\) with support over the rationals \(\frac{1}{K}\{0, \ldots, K-1\}\). We let 
\begin{equation}
\Delta^{K-1}_+ = \left\{s \in \R_{\geq 0}^{\frac{1}{K}\{0, \ldots, K-1\}} : \sum_m s_{m/K} = 1, s_{0} < 1 \right\}
\label{eq:prob-simplex}
\end{equation}
be the probability simplex with non-empty positive support. 

First, we show that the mixed condition of \cref{eq:mixed-main-cont} is sufficient if one allows an increase in the dimension of cycling. However, we also provide an example where the mixed condition holds, a higher dimensional cycle exists, but no 2D root-of-unity cycle exists, implying that \(\Oaug \subsetneq \OCycle\). 

\begin{proposition}[Sufficiency]
\label{thm:orthogonal-planes}
    Suppose \eqref{eq:mixed-main-cont} holds for some distribution \(\{s_{m/K}\}_{m=0}^{K-1} \in \Delta^{K-1}_+\) with support \(\operatorname{supp}(s) = \{m : s_{m/K} > 0\}\) with cardinality \(r = |\operatorname{supp}(s)|\). Then, there exists a cycle in dimension \(2r\). Thus, 
    \[
    \OCycle = \bigcup_{K \geq 3} \mathcal{N}_K, 
    \]
    for
    \[
    \mathcal{N}_K = \left\{ (\gamma, \beta) : \exists s \in \Delta^{K-1}_+ \text{ s.t. } \eqref{eq:mixed-main-cont} \text{ is feasible for } s. \right\} \label{eq:mixed-main-set}
    \]
\end{proposition}
\begin{proof}
    Let \(S = \{m_1, \ldots, m_r\}\). We place each mode in its own orthogonal dimension, forming the cyclic signal:
    \[
    z_t = (\sqrt{s_{m_1/K}} \omega_K^{m_1 t}, \sqrt{s_{m_2/K}} \omega_K^{m_2 t}, \ldots, \sqrt{s_{m_r/K}} \omega_K^{m_r t}) \in \Cx^r. 
    \]
    Then, the DFT coefficients become
    \[
    c_{m_i/K} = \sqrt{s_{m_i/K}} e_{i}. 
    \]
    where \(e_i\) the \(i\)-th basis vector. The coefficients are thus orthogonal with each other: \((c_{m'/K}^* c_{m/K}) = 0\), the cross-terms vanish, and by mixed feasibility the non-cross terms are controlled by constant \(0\) function values, and \cref{eq:dft-inner-product-main-cont} holds. Because the support also contains positive frequencies, the cyclic signal \(\{z_t\}\) is not trivial (i.e. not all points are equal), and we get cycling. 
\end{proof}

We note that the mixed condition \eqref{eq:mixed-main-cont} matches the linear problem obtained by \citet[Theorem 6.13]{goujaud2025provable}, while our constructed higher-dimensional cycle is in essence a block-wise root-of-unity cycle as the one of \citet[Proposition 6.11]{goujaud2025provable}. Thus, Proposition~\ref{thm:orthogonal-planes} is a rederivation of \citep[Theorem 6.13]{goujaud2025provable} using DFT instead of circulant matrices. 

We record an AI-generated example of the mixed condition holding, \(4\)-dimensional cycling being feasible, but \(2\)-dimensional root-of-unity cycling being infeasible:
\begin{equation}
 P_1=(\kappa,\beta,\gamma)
 =\left(\frac{531}{10^4},\frac{4879}{10^4},
       \frac{32829}{12500}\right).
 \label{eq:P1}
\end{equation}
For \(K=10\), the distribution
\begin{equation}
 s_{2/10}=\frac{137}{500},\qquad
 s_{3/10}=\frac{363}{500}
 \label{eq:P1weights}
\end{equation}
strictly satisfies all nine mixed inequalities.  By \cref{thm:orthogonal-planes}, placing modes \(2\) and \(3\) in separate planes gives a cycle in \(\R^4\).

We can perform an exact exclusion of every pure roots of unity mode for every cycle length \(K\).  The argument is provided in Appendix~\ref{app:root-exclusion}: all sufficiently large lengths are excluded by one uniform bound, and the remaining finitely many primitive modes are checked exactly. Subject to that certificate, one obtains
\begin{equation}
 \Oaug\subsetneq\OCycle.
 \label{eq:strict-gap}
\end{equation}

Recall the inclusions
\[
\Oaug \subseteq \OCycle^{(2)} \subseteq \ldots \subseteq \OCycle^{(d)} \subseteq \OCycle, \ d \geq 2.
\]
The strict inclusion \(\Oaug\subsetneq\OCycle\) implies that equality breaks somewhere along the chain. A natural question is the following: Does it hold that \(\Oaug \subsetneq \OCycle^{(2)}\), or \(\OCycle^{(2)} \subsetneq \OCycle\), or both. We conjecture the following: 

\begin{conjecture}[2D-restricted cycling]
\(\Oaug = \OCycle^{(2)} \subsetneq \OCycle\). In particular, when restricting the dimension to two, we conjecture that the necessary mixed condition stops being sufficient, and we have \(\Oaug = \OCycle^{(2)}\) solely from \cref{eq:dft-inner-product-main-cont}. 
\end{conjecture}

%% file: sections/connections.tex
\section{Connection to absolute stability}\label{sec:connections}

One of the earliest works highlighting the cycling behavior of heavy-ball also happens to be one of the first connecting it with control systems: \citet{lessard2016analysis} cast the update rules of several first-order methods in the form of dynamical systems with nonlinearities, and use techniques from control theory, such as integral quadratic constraints \citep{megretski1997system}, to obtain convergence guarantees for these algorithms. Despite this connection, in the realm of characterizing instability, there exists to the best of our knowledge little connection between the works of control theory and optimization. Here, we briefly make explicit the connection between heavy-ball cycling and
the periodic Lur'e-system construction of \citet{seiler2021construction}.

To connect \eqref{eq:HB} with Lur'e systems, consider the \eqref{eq:HB} recursion on a function \(f\). Let the cyclic signal \(\{z_t\}_{t\in \Z} \subseteq \Cx\) be the periodic extension of a cyclic heavy-ball trajectory given by the \eqref{eq:HB} recursion for some parameter pair \((\gamma, \beta) \in \Ocv(\Qml)\). We first find the equivalent plant \(G\) expressing the \eqref{eq:HB} recursion. With input \(u_k\), recall that the \eqref{eq:HB} recursion reads as
\[
z_{t+1} = z_t - \gamma u_t + \beta(z_t - z_{t-1}). 
\]
Applying the Z-transform (here \(z\) refers to the Z-transform variable, not the cyclic signal) to the above equation yields
\begin{gather}
z Z(z) = Z(z) - \gamma U(z) + \beta (Z(z) - z^{-1} Z(z)) \notag \\
\Leftrightarrow U(z) = \frac{(1+\beta) - z - \beta z^{-1}}{\gamma} Z(z) = -\frac{(1-z)(\beta-z)}{\gamma z} Z(z). \label{eq:iter2nonlin}
\end{gather}
This gives a form of the inverse plant, which acts like an LTI system whose impulse response in frequency domain (i.e., for \(z\) over the unit circle) is given by the unnormalized \eqref{eq:HB} operator \(\hat M(f)=\frac{(1+\beta) - e^{2\pi i f} - \beta e^{-2\pi i f}}{\gamma}\). 
The inverse system is given by the transfer function
\[
\hat G(z) = -\frac{\gamma z}{(1-z)(\beta-z)}
\]
and has poles on \(1\) and \(\beta\). The zero-frequency pole on \(1\) corresponds to the DC component and does not affect the nonzero sinusoidal modes considered in our cycling construction; because \(\beta < 1\) the second pole is indifferent. In this form, the stepsize \(\gamma\) plays the role of the gain. 

The construction of \citet{seiler2021construction} is closely related to Theorem~\ref{thm:projection-main}: both construct a one-dimensional periodic orbit from a phase-shifted sinusoid and reduce feasibility to slope restrictions on the plant (in our case, on the \eqref{eq:HB} operator). In our formulation, we work with the inverse relation \eqref{eq:iter2nonlin}, mapping the candidate iterates to the required gradients. The slope constraints are exactly the simultaneous non-negativity inequalities used in the proof of Theorem~\ref{thm:projection-main}. Hence, our obtained angle bound in Lemma~\ref{thm:primitive-rou-feasibility}(ii) used in our projected one-dimensional cycle has the same control-theoretic interpretation as the phase condition appearing in the periodic Lur'e construction of \citet{seiler2021construction}.

%% file: sections/conclusion.tex
\section{Conclusion}\label{sec:conclusion}

We studied the cycling behavior of the heavy-ball method in smooth, strongly convex minimization. We proved the non-acceleration of heavy-ball for smooth strongly convex univariate functions. More generally, we showed that if the method cycles over the roots-of-unity, then it also admits a one-dimensional cycle after a suitable rotation and projection. Regarding general cycling, we provided a higher dimensional cycling instance which is not captured by root-of-unity cycling, and we connected our analysis with the fractional sets observed by \citet{goujaud2025open}. Finally, we placed our work among existing literature of the absolute stability of Lur'e systems. We set a new conjecture regarding cycling restricted in two dimensions, namely that it is fully characterized by root-of-unity cycling, which is a possible direction for future research. 

%% file: appendix/dft-interpolation.tex
\section{Missing interpolation step in proof of Lemma~\ref{thm:primitive-rou-feasibility}}\label{app:dft-interpolation}

The skipped steps are similar to \cite{goujaud2025provable} and are the following: using the DFT decomposed signal, we have cycling if and only if
{
\footnotesize
\begin{align*}
\exists f\in \mathcal{F}_{\mu, L}(\R^2), \forall t \in [\![K]\!], \nabla f\left(\sum_{m = 0}^{K-1} c_m \omega^{mt}\right) &= \frac{(1+\beta) \sum_{m} c_m \omega^{mt} - \sum_{m} c_m \omega^{m(t+1)} - \beta \sum_{m} c_m \omega^{m(t-1)}}{\gamma} \\
&= \sum_{m} c_m \frac{(1+\beta) - \omega^{m} - \beta \omega^{-m}}{\gamma} \omega^{mt} \\
\Leftrightarrow \exists \bar{f}\in \mathcal{F}_{0, 1}(\R^2), \forall t \in [\![K]\!], \nabla \bar{f}\left(\sum_{m = 0}^{K-1} c_m \omega^{mt}\right) &= \sum_{m} c_m \frac{(1+\beta - \mu \gamma) - \omega^{m} - \beta \omega^{-m}}{(L-\mu)\gamma} \omega^{mt} \\
&= \sum_{m} c_m M_m \omega^{mt} \\
\Leftrightarrow \exists \bar{f}^*\in \mathcal{F}_{1, \infty}(\R^2), \forall t \in [\![K]\!], \nabla \bar{f}^*\left(\sum_{m} c_m M_m \omega^{mt}\right) &= \sum_{m = 0}^{K-1} c_m \omega^{mt} \\ 
\Leftrightarrow \exists \tilde{f}\in \mathcal{F}_{0, \infty}(\R^2), \forall t \in [\![K]\!], \nabla \tilde{f}\left(\sum_{m} c_m M_m \omega^{mt}\right) &= \sum_{m} c_m (1 - M_m) \omega^{mt}. 
\end{align*}
}
At this point, we apply convex interpolation to arrive at the equivalent form of the main text. 

%% file: appendix/root-exclusion.tex
\section{Exact exclusion of all pure root-of-unity modes}
\label{app:root-exclusion}

We exclude every pure mode at the point \(P_1\) of the main text. First, we exclude coverage from any \(K \geq 10\). For a pure mode at level \(K\), let \(f=m/K\) with \(\gcd(m,K)=1\).  Summing the residual inequalities for \(\Delta\) and \(K-\Delta\) gives
\[
p(f) - p(f)^2 - q(f)^2\ge0.
\]
Similarly to the proof of Lemma~\ref{thm:primitive-rou-feasibility}, after possibly reflecting about the real axis, we can assume that \(m \leq (K-1)/2\). Then, \(q(f)>0\), and we define
\[
 r(f)=\frac{p(f)(1-p(f)) - q(f)^2}{q(f)}.
\]
The \(\Delta\) satisfying \(m\Delta\equiv1\pmod{K}\) realizes the smallest positive grid angle \(2\pi/K\).  The corresponding residual inequality can be written equivalently as \begin{equation}
 r(f)\ge\frac{\sin(2\pi/K)}{1-\cos(2\pi/K)}=\cot\left(\frac{\pi}{K}\right).
 \label{eq:cot-necessary}
\end{equation}

Now \(p(1-p)-q^2\le p(1-p)\le1/4\).  Write
\begin{equation}
 p=\frac{1+\beta-\kappa\gamma}{(1-\kappa)\gamma}-\frac{1+\beta}{(1-\kappa)\gamma}u,
 \qquad q=\frac{1-\beta}{(1-\kappa)\gamma}\sqrt{1-u^2},
 \qquad u=\cos (2\pi f).
 \label{eq:pq-u}
\end{equation}
The necessary condition \(p(1-p)-q^2\geq 0\) implies \(p\in[0,1]\), which confines \(u\) to an explicit interval. Let \(\sin_{\min}\) be the minimum of \(\sqrt{1-u^2}\) on that interval.  Then
\begin{equation}
 r(\theta)
 \le\frac{1}{4\frac{1-\beta}{(1-\kappa)\gamma}\sin_{\min}}.
 \label{eq:r-upper}
\end{equation}
At the rational point \(P_1\), to calculate \(\sin_{\min}\) we have the following:
\begin{gather*}
p=\frac{168555301 - 185987500 u}{310857801} \in [0, 1] \\
\Leftrightarrow -\frac{56921}{74395} \leq u \leq \frac{168555301}{185987500}. 
\end{gather*}
The minimum of \(\sqrt{1-u^2}\) over this interval occurs at the endpoint having the largest absolute value, which is the right endpoint. Hence, 
\[
\sin_{\min}^2 = \frac{6180460661049399}{34591350156250000}.
\]
In addition, we have that \(\cot^2(\pi/10) = 5+2\sqrt{5}\). Thus, we obtain the inequality
\begin{equation}
 \cot^2\left(\frac{\pi}{10}\right)
 -\left(\frac{1}{4\frac{1-\beta}{(1-\kappa)\gamma}\sin_{\min}}\right)^2
 =2\sqrt5-
 \frac{85559905317962889383}{26331249736617936912}>0.
 \label{eq:exact-exclusion-number}
\end{equation}
Since \(\cot(\pi/K)\) increases with \(K\), \eqref{eq:cot-necessary} is impossible for every \(K\ge10\).  The remaining primitive modes with \(K\le9\) form a finite list and were machine-checked exactly. We provide Table~\ref{tab:pure-mode-exclusion} for these cases. 

\begin{table}[t]
    \centering
    \caption{Exhaustive verification of all primitive pure modes with
    $3 \leq K \leq 9$, up to the reflection symmetry $m \leftrightarrow K-m$.
    Here $\Delta_m$ is the unique integer in $\{1,\dots,K-1\}$ satisfying
    $m\Delta_m \equiv 1 \pmod K$, and
    $A_{K,m} := p_{K,m}(1-p_{K,m}) - q_{K,m}^2$.
    In every case,
    $A_{K,m} - q_{K,m}\cot(\pi/K)<0$, equivalently
    $\phi_{K,m}(\Delta_m)>0$, so the corresponding pure mode is infeasible.}
    \label{tab:pure-mode-exclusion}
    \small
    \begin{tabular}{@{}rrrrrr@{}}
        \toprule
        $K$
        & $m$
        & $\Delta_m$
        & $A_{K,m}$
        & $q_{K,m}\cot(\pi/K)$
        & $A_{K,m}-q_{K,m}\cot(\pi/K)$ \\
        \midrule
        3 & 1 & 1 &  0.101657820 & 0.102961064 & $-0.001303244$ \\
        4 & 1 & 1 &  0.205813010 & 0.205922128 & $-0.000109119$ \\
        5 & 1 & 1 &  0.191293486 & 0.269555566 & $-0.078262080$ \\
        5 & 2 & 3 & $-0.041604599$ & 0.166594501 & $-0.208199100$ \\
        6 & 1 & 1 &  0.152186257 & 0.308883193 & $-0.156696936$ \\
        7 & 1 & 1 &  0.114644817 & 0.334312475 & $-0.219667658$ \\
        7 & 2 & 4 &  0.178944050 & 0.416880838 & $-0.237936788$ \\
        7 & 3 & 5 & $-0.095868914$ & 0.185529427 & $-0.281398340$ \\
        8 & 1 & 1 &  0.083760067 & 0.351531062 & $-0.267770994$ \\
        8 & 3 & 3 &  0.012302067 & 0.351531062 & $-0.339228995$ \\
        9 & 1 & 1 &  0.059339548 & 0.363667631 & $-0.304328082$ \\
        9 & 2 & 5 &  0.205071765 & 0.557171135 & $-0.352099370$ \\
        9 & 4 & 7 & $-0.120318113$ & 0.193503504 & $-0.313821617$ \\
        \bottomrule
    \end{tabular}
\end{table}